\documentclass[letterpaper,12pt,conference]{ieeeconf}  

\IEEEoverridecommandlockouts                              
\usepackage{times} 

\usepackage{amsmath}
\usepackage{amssymb}
\usepackage{amsfonts}
\usepackage{graphicx}

\newtheorem{lemma}{Lemma}
\newtheorem{theorem}{Theorem}
\newtheorem{proposition}{Proposition}

\def\<{\leqslant}           
\def\>{\geqslant}           
\def\div{{\rm div}}         

\def\mR{{\mathbb R}}    
\def\mC{{\mathbb C}}    

\def\Tr{{\rm Tr}}       
\def\rT{{\rm T}}        

\def\bP{{\mathbf P}}    
\def\bE{{\mathbf E}}    

\def\bra{{\langle}}
\def\ket{{\rangle}}

\def\Bra{\left\langle}
\def\Ket{\right\rangle}

\def\re{{\rm e}}        
\def\rd{{\rm d}}        
\def\d{\partial}        

\def\cL{{\mathcal L}}

\def\x{\times}

\def\mG{{\mathbb G}}

\def\wh{\widehat}
\def\wt{\widetilde}

\def\mP{{\mathbb P}}

\def\bD{{\bf D}}

\def\cF{{\cal F}}
\def\cW{{\mathcal W}}

\def\cM{{\mathcal M}}

\def\cC{{\mathcal C}}

\def\cov{{\bf cov}}

\def\cN{{\cal N}}

\def\bL{{\bf L}}
\def\bM{{\bf M}}

\def\mS{{\mathbb S}}

\def\Ups{\Upsilon}
\def\phi{\varphi}

\begin{document}
\title{\Large\bf
Minimum Relative Entropy State Transitions in Linear Stochastic Systems: the Continuous Time Case
}
\author{
    Igor G. Vladimirov,\qquad Ian R. Petersen
\thanks{This work is supported by the Australian Research Council. The authors are with the School of Engineering and Information Technology, University of New South Wales at the  Australian  Defence Force Academy, Canberra ACT 2600, Australia. E-mail: {\tt igor.g.vladimirov@gmail.com, i.r.petersen@gmail.com}.}
}
\onecolumn
\pagestyle{plain}
\maketitle
\thispagestyle{empty}

\begin{abstract}
This paper develops a dissipativity theory  for dynamical systems governed by linear It$\hat{\bf o}$  stochastic differential equations driven by random noise with an uncertain drift. The deviation of the noise from a standard Wiener process in the nominal model is quantified by relative entropy. The paper discusses a dissipation inequality for the noise relative entropy supply. The problem of minimizing the supply required to drive the system between given Gaussian state distributions over a specified time horizon is considered. This problem, known in the literature as the Schr\"{o}dinger bridge, was treated previously in the context of reciprocal processes. The paper obtains a closed-form smooth solution to a  Hamilton-Jacobi equation for the minimum required relative entropy supply by using nonlinear algebraic techniques.
\end{abstract}


%
%
%
%
%

\section{Introduction}

We consider a dynamical system whose state is a diffusion process governed by a linear It$\hat{\rm o}$ stochastic differential equation (SDE) driven by a random noise. The noise is generated from a standard Wiener process by another SDE with an uncertain drift. The case where the drift vanishes and the noise replicates the Wiener process, represents the {\it nominal} scenario. A nonzero drift in the noise SDE can be interpreted as the strategy of a hypothetical player who uses the past history of the system state in order to move its probability density function (PDF) away from the {\it nominal invariant state PDF}. The deviation of the actual noise distribution from the Wiener measure  can be quantified by the Kullback-Leibler relative entropy \cite{CT_2006}. As a measure of uncertainty in the noise distribution, the relative entropy is often utilized  in the robust control of stochastic systems \cite{CR_2007,DJP_2000,PUS_2000,UP_2001}.

The {\it noise relative entropy} over a bounded time interval can be regarded as a stochastic analogue of the supply which is a fundamental concept in the theory of deterministic dissipative systems \cite{Willems_1972}. This analogy leads to a {\it dissipation inequality} which links the noise relative entropy supply with the increment in the relative entropy of the state PDF of the system with respect to the nominal invariant state PDF. The {\it state relative entropy}, therefore, plays the role of a {\it  storage function}. The relative entropy dissipation inequality is related to Jarzynski's equality \cite{Jarzynski_1996} for the Helmholtz free energy in open dynamical systems. This non-equilibrium thermodynamics viewpoint, where the noise results from interaction of the system with its surroundings (via mechanical work and heat transfer), motivates a stochastic dissipativity theory in the form of a variational problem  involving  entropy. Such problems are more complex than their deterministic counterparts since they deal with probability measures (or PDFs) on signal spaces, rather than the signals themselves.

We are mainly concerned with computing the minimum noise relative entropy supply required to drive the system between given initial and terminal state PDFs over a specified time horizon. The {\it state PDF transition problem}, known as the Schr\"{o}dinger bridge, was treated previously in a context of reciprocal processes (Markov random fields on the time axis) \cite{Beghi_1994,Blaqiere_1992,DaiPra_1991,Mikami_1990}. This problem was also studied for quantum systems \cite{BFP_2002}, using the formalism of stochastic mechanics \cite{Nelson_2001}.  The solution of the Schr\"{o}dinger bridge problem is related to two coupled integral equations \cite[Definition 2.3 on p.~26]{Mikami_1990} and is not available in closed form  for a general diffusion model.

We consider the state PDF transition problem with Gaussian initial and terminal state PDFs and undertake a different, somewhat more algebraic, approach.   Using Markovization and stochastic linearization of the noise strategy as entropy-decreasing operations, we establish a mean-covariance separation principle which splits the minimum required noise relative entropy supply into two independent terms associated with the mean and covariance matrix of the system state. While the \textit{mean} part  is calculated using standard linear quadratic optimization,   the \textit{covariance} part (which  is a function of matrices) satisfies a Hamilton-Jacobi equation (HJE) complicated by inherent noncommutativity.

This partial differential equation (PDE) has a quadratic Hamiltonian on its right-hand side (with the quadraticity coming from the diffusion part of the system dynamics) and involves a certain boundary condition. The bilinear ``interaction'' of solutions of this PDE (which, in the quadratic case, replaces the superposition principle) allows them to be generated in a quasi-additive way. Unlike infinitesimal perturbation techniques based on asymptotic expansions in the small noise limit, our approach provides a {\it finite} correction scheme which allows a closed-form smooth solution to be found for the HJE.

The correction scheme also employs  an ansatz class of ``trace-analytic'' functions of matrices and a  matrix version of the separation of variables which not only copes with the nonlinearity but also essentially ``scalarizes''  the covariance HJE, thus overcoming the noncommutativity issues.
These nonlinear algebraic techniques may therefore be of interest in their own right from the viewpoint of nonlinear PDEs and holomorphic functional calculus.

\section{Class of systems being considered\label{sec:system}}

We consider a  dynamical system whose state $X:=(X_t)_{t\> 0}$ is a diffusion process in $\mR^n$
governed by an
It$\hat{\rm o}$ SDE
\begin{equation}
\label{XW}
    \rd X_t
    =
    f(X_t)\rd t
    +
    B
    \rd W_t,
    \qquad
    f(x):= \mu+Ax,
\end{equation}
driven by an $\mR^m$-valued random noise
$W:=(W_t)_{t\> 0}$.
Here, $\mu \in \mR^n$, $A\in \mR^{n\x n}$, $B \in \mR^{n\x m}$, with
$A$ Hurwitz and ${\rm rank} B = n\< m$, so that the diffusion matrix
\begin{equation}
\label{DBB}
    D := BB^{\rT}
\end{equation}
is positive definite. The noise $W$ is an
It$\hat{\rm o}$ process interpreted as an external random noise which originates from interaction of the
system with its environment
and is generated by another SDE
\begin{equation}
\label{WV}
    \rd W_t
    =
    h_t \rd t + \rd \cW_t,
\end{equation}
with an uncertain drift $h := (h_t)_{t\> 0}$.
Here, $h$ is a
random process with values in $\mR^m$, adapted to the natural filtration
$(\cF_t)_{t\> 0}$ of $X$, where $\cF_t$ is the
$\sigma$-subalgebra of events induced by the history $X_{[0,t]}$ of
 $X$ on the time interval $[0,t]$. Also,
 $\cW:= (\cW_t)_{t\> 0}$ is an $m$-dimensional standard Wiener process, independent of  $X_0$.
  Substituting (\ref{WV}) into (\ref{XW}) yields
\begin{equation}
\label{gaussXV}
    \rd X_t
    =
    (f(X_t) + Bh_t)\rd t
    +
    B
    \rd \cW_t.
\end{equation}
We assume that $\bE\int_0^t |h_s|^2\rd s < +\infty$ for all $t>0$. Together with $D\succ 0$, the local mean square integrability of $h$ ensures the absolute continuity of the system state  $X_t$.
 The case $h\equiv 0$ (where $W \equiv \cW$) represents the {\it nominal} scenario of the system-environment interaction. 
   A nonzero drift $h_t$ is interpreted as the strategy of a hypothetical player who uses the past history $X_{[0,t]}$  of the system state $X_t$ to move the {\it state PDF} $p_t$ away from the \textit{nominal invariant state PDF}
\begin{equation}
\label{gauss_p*}
    p_*(x)
    =
    (2\pi)^{-n/2}(\det \Pi_*)^{-1/2}
    \exp
    (
        -
        \|x-\alpha_*\|_{\Pi_*^{-1}}^2/2
    ),
\end{equation}
which the system would have in the nominal case. The nominal invariant state distribution is Gaussian, $\cN(\alpha_*, \Pi_*)$,  with mean $\alpha_*$ and covariance matrix $\Pi_*$ given by
\begin{equation}
\label{alphaPi*}
    \alpha_*
    =
    -A^{-1}\mu,
    \qquad
    \Pi_*
    =
    \int_{0}^{+\infty}
    \re^{At}
        D
    \re^{A^{\rT} t}
    \rd t,
\end{equation}
where $\Pi_*$ is  the infinite-horizon controllability
Gramian of the pair $(A, B)$ satisfying the algebraic Lyapunov
equation
\begin{equation}
\label{Pi*}
    A\Pi_* + \Pi_* A^{\rT} +D
    =
    0.
\end{equation}
The deviation of the actual noise distribution from
the Wiener measure is quantified by
\begin{equation}
\label{E}
    E_t
     :=
    \bD(\bP_t \| \bP_t^*)
    -
    \bD(\bP_0 \| \bP_0^*)
    =
        \frac{1}{2}
        \int_{0}^{t}
        \bE(|h_s|^2)\rd s.
\end{equation}
Here, Girsanov's theorem \cite{Girsanov_1960} is used;  $\bD(M\|N):= \bE_M\ln(\rd M/\rd N)$ is the Kullback-Leibler
relative entropy \cite{CT_2006} of a probability measure $M$ with respect to another probability measure $N$ (under the assumption of absolute continuity $M\ll N$); and $\bP_t$ and $\bP_t^*$ are the
restrictions of the true and nominal probability measures $\bP$ and
$\bP_*$ to the $\sigma$-algebra $\sigma(X_0, W_{[0,t]})$. The expectation  $\bE$ in (\ref{E}) is over
$\bP$ under which the noise $W$, governed by
(\ref{WV}), becomes a standard Wiener process if and only if
$h\equiv 0$. The {\it noise relative entropy} $E_t$ over the time
interval $[0,t]$ from (\ref{E}) can be regarded as a stochastic
counterpart of the {\it supply} in
the theory of deterministic dissipative systems \cite{Willems_1972}.

In the nominal case $h\equiv 0$,
the state $X$  of the system is a
homogeneous Markov diffusion process and the state PDF $p_t$
satisfies the Fokker-Planck-Kolmogorov equation (FPKE)
\begin{equation}
\label{FPK}
    \d_t p_t
    =
    \cL^{\dagger}(p_t),
\end{equation}
where
\begin{equation}
\label{cL+}
    \cL^{\dagger}(p)
    :=
    \div^2(Dp)/2    -\div (fp).
\end{equation}
Here, for any twice continuously differentiable function $G :=
(G_{ij})_{1\<i,j\< n}: \mR^n \to \mS_n$, with $\mS_n$ the space of real symmetric matrices of order $n$, the maps $\div G: \mR^n \to
\mR^n$ and $\div^2 G: \mR^n \to \mR$ are defined by
$    \div G
    :=
    (
        \sum_{j=1}^{n}
        \nabla_j G_{ij}
    )_{1\< i\< n}
$
and $
    \div^2 G
    :=
    \div \div G
    =
    \sum_{1\< i, j \< n}
    \nabla_i \nabla_j G_{ij}
$,
where $\nabla_i:= \d_{x_i}$ is the partial derivative with respect to  the $i$th
Cartesian coordinate in $\mR^n$. 
The operator
$\cL^{\dagger}$ in (\ref{cL+}) is the formal adjoint of the
infinitesimal generator $\cL$ of $X$
in the nominal case. The action of $\cL$ on a twice
continuously differentiable test function $\phi: \mR^n \to \mR$ with
bounded support is described by
\begin{equation}
\label{cL}
    \cL(\phi)
    =
    f^{\rT}\nabla \phi
    +
    \Tr(D \phi'')/2,
\end{equation}
where $(\cdot)''$  is the Hessian
matrix. Since $A$ is Hurwitz, the system is ergodic under the nominal noise $W=\cW$
and the PDF (\ref{gauss_p*}) is a steady-state solution of the FPKE
(\ref{FPK}): $\cL^{\dagger}(p_*)\equiv 0$. 
The controllability of $(A,B)$ (which follows from $D \succ 0$) is equivalent to
$\Pi_*\succ 0$, and is also equivalent to the
nonsingularity of the finite-horizon controllability Gramian
\begin{equation}
\label{Gamma}
    \Gamma_t
    :=
    \int_{0}^{t}
    \re^{As} D \re^{A^{\rT}s}
    \rd s
    =
    \Pi_* - \re^{At} \Pi_* \re^{A^{\rT}t}
\end{equation}
for any $t>0$. We define two semigroups of affine transformations  $(M_t)_{t \> 0}$ and
$(C_t)_{t\> 0}$ by
\begin{align}
\label{M}
    M_t(\alpha)
     &:= 
    \re^{At} \alpha + A^{-1}(\re^{At}-I_n) \mu
    =
    \alpha_* + \re^{At} (\alpha-\alpha_*),\\
 \label{C}
    C_t(\Sigma)
     &:=  
    \re^{At} \Sigma \re^{A^{\rT}t} + \Gamma_t
    =
    \Pi_* + \re^{At} (\Sigma - \Pi_*)\re^{A^{\rT}t}.
\end{align}
These semigroups act on $\mR^n$ and the set $\mS_n^+$ of real positive semi-definite symmetric matrices of order $n$ and describe the nominal
evolution of the state mean and covariance matrix
\begin{equation}
\label{alphaPi}
    \alpha_t
    :=
    \bE X_t,
    \qquad
    \Pi_t
    :=
    \cov(X_t).
\end{equation}
The infinitesimal generators of the semigroups are given by
\begin{equation}
\label{MCdot}
    \cM(\alpha)
     =
    \mu + A \alpha,
    \qquad
     \cC(\Sigma)
      =
     A\Sigma + \Sigma A^{\rT} + D.
\end{equation}
In general (when $h\not\equiv 0$), the linearity of the SDE
(\ref{gaussXV}) allows the dynamics of (\ref{alphaPi})  to be described by
\begin{equation}
\label{alphaPidot}
    \dot{\alpha}_t
      =
     \cM(\alpha_t)
     +
     B \beta_t,
     \qquad
    \dot{\Pi}_t
      =
     \cC(\Pi_t)
     +
     BK_t\Pi_t + \Pi_t K_t^{\rT} B^{\rT}
\end{equation}
in terms of the moments
\begin{equation}
\label{betaK}
    \beta_t
    :=
    \bE h_t,
    \qquad
    K_t
    :=
    \cov(h_t,X_t) \Pi_t^{-1}.
\end{equation}

\section{Markovization and state PDF dynamics\label{sec:state_PDF_evolution}}

For any $t\> 0$, we define a function  $\overline{h}_t: \mR^n \to
\mR^m$ associated with the noise strategy $h$ by
\begin{equation}
\label{hover}
    \overline{h}_t(x)
    :=
    \bE(h_t | X_t = x).
\end{equation}
In particular, if
$h_t$ is a deterministic function of $t$ and the
current state $X_t$, then
\begin{equation}
\label{hmarkov}
    h_t
    =
    \overline{h}_t(X_t),
    \qquad
    t\> 0.
\end{equation}
The noise strategies $h$, satisfying (\ref{hmarkov}) with probability one, are said to be Markov  with respect to the state of the system.

\begin{proposition}
\label{prop:state_PDF_evolution} Suppose that the state PDF $p_t(x)$
of the system, governed by (\ref{gaussXV})  is continuously
differentiable in $t>0$ and twice continuously differentiable
in $x \in \mR^n$. Then it satisfies the FPKE
\begin{equation}
\label{FPKh}
    \d_t p_t
    = 
    \div^2(Dp_t)/2-\div ((f + B \overline{h}_t) p_t)
=
    \cL^{\dagger}(p_t)
    -
    \div(B\overline{h}_t p_t),
\end{equation}
where the operator $\cL^{\dagger}$ is defined by (\ref{cL+}).
\end{proposition}
\begin{proof} In view of the smoothness of $p_t$ and  the identity
$
    \bE(f(X_t) + B h_t | X_t = x)
    =
    f(x) + B\overline{h}_t(x)
$
which follows  from (\ref{hover}),  the PDE (\ref{FPKh}) is
obtained from the weak formulation of the FPKE for It$\hat{\rm
o}$ processes in \cite[Eqs. (0.11)--(0.13) on p.~21]{Mikami_1990}.
\end{proof}

The PDE (\ref{FPKh}) governs the PDF of a
Markov diffusion process $\xi:= (\xi_t)_{t\> 0}$ generated by the
SDE
$    \rd \xi_t
    =
    (f(\xi_t) + B \overline{h}_t(\xi_t))
    \rd t
    +
    B \rd \cW_t
$.
If $\xi_0$ and $X_0$ are identically distributed with the state PDF $p_0$, then
$\xi_t$ and $X_t$ share the common
PDF $p_t$ for any $t>0$.  The passage $h_t\mapsto \overline{h}_t(X_t)$ from an arbitrary noise strategy $h$ to the Markov
strategy,  defined by
(\ref{hover}), is referred to as the {\it Markovization} of $h$
and denoted by $\bM$. Although $\bM$ preserves the state PDFs
$p_t$ of the system, the multi-time probability
distributions of $X$
are, in general, modified. They all remain unchanged under the Markovization $\bM$ if and only if $h$ is Markov.
 Since such strategies are invariant under $\bM$, the Markovization
is idempotent: $\bM^2 = \bM$.

\begin{theorem}
\label{th:markovization} The Markovization $\bM$ of a noise strategy $h$ does not increase
the noise relative entropy supply (\ref{E}):
\begin{equation}
\label{EMarkov}
    E_T
    \>
    \frac{1}{2}
    \int_{0}^{T}
    \bE
    (
        |\overline{h}_t(X_t)|^2
    )
    \rd t,
\end{equation}
where $\overline{h}_t$ is given by (\ref{hover}).
This inequality is an equality if and only if $h$ is Markov in the sense of (\ref{hmarkov}).
\end{theorem}
\begin{proof}
The standard properties of iterated conditional expectations, strict convexity of the squared Euclidean norm $|\cdot|^2$ and Jensen's inequality imply that
$    \bE (|h_t|^2)
     =
    \bE \bE (|h_t|^2 | X_t)
     \>
    \bE (|\overline{h}_t(X_t)|^2)
$,
where the inequality becomes an equality if and only
if (\ref{hmarkov}) holds with probability one.
Integration over $t\in [0,T]$
yields (\ref{EMarkov}) whose
right-hand side is completely specified by the
functions $\overline{h}_t$ and $p_t$. Since the state PDFs
remain unchanged under the Markovization of the noise strategy,
then (\ref{EMarkov}) holds as an equality if and only if
$h$ is Markov.
\end{proof}

\section{Relative entropy dissipation inequality\label{sec:dissipation}}

We will now consider the {\it state relative entropy} defined by
\begin{equation}
\label{R}
    R_t :=
    \bE
    \ln
    q_t(X_t)
    =
    \Bra
        p_t,
        \ln q_t
    \Ket,
\end{equation}
where $
    \bra
        a, b
    \ket
    :=
    \int_{\mR^n}
    a(x)^{\rT}b(x)
    \rd x
$ is the inner product of functions $a,b:\mR^n \to
\mR^r$ (provided the integral exists, as is the case, for example, when $a$, $b$ are square integrable), and
\begin{equation}
\label{q}
    q_t
    :=
    p_t/p_*
\end{equation}
is the true-to-nominal state PDF ratio.
Note the difference between
$R_t$ and the noise relative entropy $E_t$ defined in
(\ref{E}); see \cite[Eq.~(3.12) \&
Remark on p.~321]{DaiPra_1991}. In the nominal case,
 $R_t$ is non-increasing
in $t$, which represents the Second Law of
Thermodynamics for homogeneous Markov processes as models of
isolated systems \cite{CT_2006}. A  nonzero $h$ makes (\ref{gaussXV}) an open system and $R_t$ is no longer monotonic.

\begin{theorem}
\label{th:dissipation} For any $T>0$, the increment $R_T - R_0$ of
the state relative entropy (\ref{R}) satisfies
\begin{equation}
\label{dissipation}
    R_T
    -
    R_0
     \<
    E_T
    -
    \frac{1}{2}
    \int_{0}^{T}
    \bra
        p_t,
        |
            \overline{h}_t
            -
            B^{\rT}
            \nabla \ln q_t
        |^2
    \ket
    \rd t,
\end{equation}
where $\overline{h}_t$ and $q_t$ are defined by
(\ref{hover}) and (\ref{q}). This inequality
is an equality if and only if $h$ is a Markov  noise strategy.
\end{theorem}
\begin{proof} Differentiation of the right-hand side of
(\ref{R}) gives
\begin{equation}
\label{diffR1}
    \d_t R_t
    =
    \bra
        \d_t p_t,
        \ln q_t
    \ket
    +
    \bra
        p_t,
        p_t^{-1}\d_t p_t
    \ket
    =
    \bra
        \cL^{\dagger}(p_t)
        -\div (g \overline{h}_t p_t),
        \ln q_t
    \ket.
\end{equation}
Here, use is made of (\ref{FPKh}) and
the identities $\d_t \ln q_t = p_t^{-1}\d_t p_t$ and $\bra 1, \d_tp_t\ket = 0$ which follow
from (\ref{q}) and $\bra 1, p_t\ket =
1$. Integration by parts reduces (\ref{diffR1}) to
\begin{equation}
\label{diffR2}
    \d_t R_t
     =
    \bra
        p_t,
        \cL(\ln q_t)
    \ket
    +
    \bra
        g\overline{h}_t p_t,
        \nabla \ln q_t
    \ket.
\end{equation}
In view of Fleming's logarithmic transformation \cite{Fleming_1982}
(see, also, \cite[Eq.~(81) on p.~201]{BFP_2002}), the operator
$\cL$, defined by (\ref{cL}), acts on the logarithm of a twice
continuously differentiable function $\psi: \mR^n \to (0,+\infty)$
as
$
\nonumber
    \cL(\ln \psi)
     =
    \cL(\psi)/\psi
    -
    \|\nabla \ln \psi\|_D^2/2
$.
Application of this relation to $\psi := q_t$ from (\ref{q})   represents the first inner product in  (\ref{diffR2}) in the form
\begin{align}
\nonumber
    \bra
            p_t,
            \cL(\ln q_t)
    \ket
     & =
    \bra
        p_t,
        \cL(q_t)/q_t
        -
        \|\nabla \ln q_t\|_D^2/2
    \ket\\
\label{diffR3}
     & =
    \bra
            p_*,
            \cL(q_t)
    \ket
    -
    \bra
            p_t,
            \|
                \nabla \ln q_t
            \|_D^2
    \ket/2
     =
    -
    \bra
            p_t,
            \|
                \nabla \ln q_t
            \|_D^2
    \ket/2,
\end{align}
where $\bra p_*, \cL(q_t)\ket = \bra\cL^{\dagger}(p_*), q_t\ket = 0$.
By substituting (\ref{diffR3}) into
(\ref{diffR2}), it follows that
\begin{equation}
\label{diffR4}
    \d_tR_t
      =
    \bra
        p_t,
        \overline{h}_t^{\rT} B^{\rT}\nabla \ln q_t
        -
        \|\nabla \ln q_t\|_D^2/2
    \ket
      =
    \bra
        p_t,
        |\overline{h}_t|^2
    -
        |
            \overline{h}_t
            -
            B^{\rT} \nabla \ln q_t
        |^2
    \ket
/2,
\end{equation}
where the square is completed using (\ref{DBB}). Integration of
both parts of (\ref{diffR4}) in $t$ over $[0,T]$ yields
\begin{equation}
\label{diffR5}    
R_T
    -
    R_0
    +
    \frac{1}{2}
    \int_{0}^{T}
    \bra
        p_t,
        |
        \overline{h}_t
        -
        B^{\rT}
        \nabla \ln q_t
        |^2
    \ket
    \rd t
=
    \frac{1}{2}
    \int_{0}^{T}
    \bE
    (
        |\overline{h}_t(X_t)|^2
    )
    \rd t
    \<
    E_T,
\end{equation}
where the inequality proves (\ref{EMarkov}). Now,
(\ref{dissipation}) follows from (\ref{diffR5}). The
claim that (\ref{dissipation})
holds as an equality if and only if $h$ is Markov, follows from the second part of
Theorem~\ref{th:markovization}.
\end{proof}

The relation (\ref{dissipation}) can be regarded as a {\it
dissipation inequality} \cite[pp.~327, 348]{Willems_1972}, with  the state relative entropy (\ref{R}) playing the role
of a {\it storage function}.
Thus, the absolute continuity of the state distribution of the system can
only be destroyed within the finite time $T$ using an infinite noise
relative entropy supply $E_T$. Note also that  (\ref{dissipation}),  which becomes an
equality for Markov noise strategies, can be thought of as an
analogue  of Jarzynski's equality from nonequilibrium thermodynamics
\cite{Jarzynski_1996}.

\section{Gaussian state PDF transition\label{sec:gauss_transition}}

We will now consider the problem of driving the system (\ref{gaussXV})
from a Gaussian initial state PDF $p_0 := \sigma$ to a
Gaussian terminal state PDF $p_T := \theta$ at a specified time
 $T > 0$  so as to minimize the supply
(\ref{E}) over the interval $[0,T]$:
\begin{equation}
\label{J}
    J_T(\sigma,\theta)
    :=
    \inf
    \{
        E_T:\
        p_0 = \sigma,\
        p_T = \theta
    \}.
\end{equation}
Here,
\begin{equation}
\label{gauss_start_finish}
    \sigma
    \sim
    \cN(\alpha_0, \Pi_0),
    \qquad
    \theta
    \sim
    \cN(\alpha_T, \Pi_T),
\end{equation}
with $\Pi_0, \Pi_T\succ 0$.
The required
supply (\ref{J}) vanishes if and only if $\theta$ is {\it nominally reachable} from $\sigma$ in time $T$ in the sense that $\theta = \re^{\cL^{\dagger}T}(\sigma)$. Here, $\re^{\cL^{\dagger}T}$ is the linear integral
operator (with a Markov transition kernel) which relates the
terminal state PDF $p_T$ of the system at time $T$ with the initial
state PDF $p_0$ under the nominal FPKE (\ref{FPK}). In
particular,  $
    J_T(p_*, p_*)
    =
    0
$ and, more generally,
$    J_T
    (
        \cN(\alpha,\Sigma),\,
        \cN(M_T(\alpha), C_T(\Sigma))
    ) = 0
$.
In (\ref{J}), the intermediate state PDFs $p_t$, with $0<t<T$, are not
required to be Gaussian. Nevertheless, we can
restrict attention to noise strategies which are not
only Markov, but are also affine with respect to the
state of the system, thus making the intermediate state PDFs also
Gaussian; see Fig.~\ref{fig:gauss_path}.
\begin{figure}[htb]
\begin{center}
\includegraphics[width=7cm]{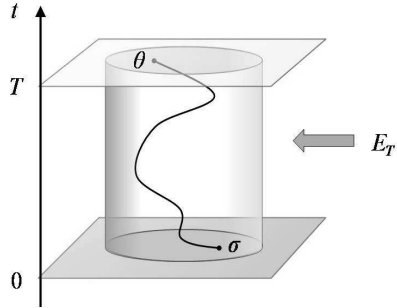}
\end{center}
\caption{ 
    \label{fig:gauss_path}
    The cylinder represents the set $\mG \x [0,T]$ whose base $\mG$ is formed by
    nonsingular Gaussian
    PDFs in $\mR^n$. The minimization of the noise relative entropy supply
    $E_T$ in driving the system state PDF between $\sigma, \theta\in
    \mG$ in time $T$
     can be restricted  to affine Markov noise strategies without affecting the minimum. Such
     strategies
     generate paths of Gaussian state PDFs from $\sigma$ to
    $\theta$ which are entirely contained by the cylinder.
}
\end{figure}

\section{Mean-covariance separation principle \label{sec:meancovseparation}}

For any $T>0$, we define a function $S_T: \mP_n^2
\to \mR_+$, with $\mP_n$ the set of real positive definite symmetric matrices of order $n$, by
\begin{equation}
\label{S}
    S_T(\Sigma, \Theta)
    :=
    \inf
        \int_{0}^{T}
        \Tr( K_t \Pi_t K_t^{\rT})
        \rd t.
\end{equation}
Here, the minimization is over $\mR^{m\x n}$-valued functions
$K_t$ from (\ref{betaK}) such that the state covariance
matrix $\Pi_t$, governed by the second of the ordinary differential equations (ODEs)  (\ref{alphaPidot}) and initialized at
$\Pi_0 := \Sigma$, satisfies the terminal condition $\Pi_T =
\Theta$. Since $C_T(\Sigma)$ is reachable from $\Sigma$ in
time $T$ with $K_t \equiv 0$ by the action of the nominal
state covariance semigroup (\ref{C}),  then
\begin{equation}
\label{S0}
    S_T(\Sigma, C_T(\Sigma))
    =
    0.
\end{equation}
\begin{theorem}
\label{th:split} The minimum required noise relative entropy supply in (\ref{J})--(\ref{gauss_start_finish})
can be computed as
\begin{equation}
\label{Jsplit}
    J_T(\sigma, \theta)
    =
    (
        \|\alpha_T - M_T(\alpha_0)\|_{\Gamma_T^{-1}}^2
        +
        S_T(\Pi_0, \Pi_T)
    )/2,
\end{equation}
where (\ref{Gamma}),  (\ref{M}), (\ref{S}) are used. The optimal noise strategy is an
affine function of the current state of the system:
\begin{equation}
\label{hlinear}
    h_t
    =
    \beta_t + K_t (X_t-\alpha_t),
\end{equation}
where $K_t$ is a function delivering the minimum in
 (\ref{S}), and
\begin{equation}
\label{betaoptimal}
    \beta_t
    :=
    B^{\rT}
    \re^{A^{\rT}(T-t)}
    \Gamma_T^{-1}
    (\alpha_T - M_T(\alpha_0)),
    \qquad
    0\< t \< T,
\end{equation}
\end{theorem}
\begin{proof}
It follows from (\ref{alphaPidot}) that the state mean
$\alpha_t$  is completely
specified by the initial condition $\alpha_0$ and the function
$\beta_t$ from (\ref{betaK}). Similarly, the state covariance matrix $\Pi_t$ is
completely specified by $\Pi_0$ and the function $K_t$. The relative entropy supply (\ref{E}) affords the lower bound:
\begin{equation}
\label{Ebound}
    2E_T
    =
    \int_{0}^{T}
    \bE(|h_t|^2)
    \rd t
    \>
    \int_{0}^{T}
    (
        |\beta_t|^2
        +
        \Tr(K_t\Pi_t K_t^{\rT})
    )
    \rd t.
\end{equation}
Indeed, by \cite[Theorem~7.7.7 on p.~473]{HJ_2007},  the Schur
complement of the block $\cov(X_t)$ in the joint covariance matrix 
of $X_t$ and $h_t$ is positive semi-definite. Hence,  $
    \cov(h_t)
    \succcurlyeq
    \cov(h_t, X_t)
    (\cov(X_t))^{-1}
    \cov(X_t, h_t)
    =
    K_t \Pi_t K_t^{\rT}
$ in view of (\ref{alphaPi})
and (\ref{betaK}), and
$    \bE(|h_t|^2)
    =
    |\beta_t|^2 + \Tr \cov(h_t)
    \>
    |\beta_t|^2 + \Tr(K_t \Pi_t K_t^{\rT})
$,
which implies (\ref{Ebound}). This inequality
becomes an equality if and only if $h_t$ is related  to $X_t$ by the
affine map (\ref{hlinear}) with probability one.  
In view of
(\ref{alphaPi}) and (\ref{betaK}),  the passage from the original
noise strategy $h$ to the right-hand side of (\ref{hlinear})
describes a {\it stochastic linearization} of $h$. We
denote this operation by $\bL$. By construction, it is idempotent:
$\bL^2 = \bL$. The stochastic linearization $\bL$ yields an affine
noise strategy which is not only Markov in the sense of (\ref{hmarkov}), but also
preserves the first two moments of $X$. Under the affine
noise strategy, the state process $X$ is Gaussian, provided that the
initial state $X_0$ is Gaussian. Thus, if $h$
is a noise strategy which drives the state PDF of the system from
$\sigma$ to $\theta$, described by (\ref{gauss_start_finish}), in
time $T$, then  $\wh{h}:= \bL(h)$ is an affine Markov strategy
under which $X$ has the same
mean and covariance matrix as it does under $h$. By the latter
property, $\wh{h}$ also drives the state PDF of the system to the
Gaussian PDF $\theta$, but supplying the same or smaller noise
relative entropy to the system over $[0,T]$ in view of (\ref{Ebound}). Therefore,  the minimization of $E_T$ can be reduced,
without affecting the minimum value, to a minimization over affine Markov noise
strategies (\ref{hlinear}). Hence, recalling
(\ref{gauss_start_finish}) and (\ref{S}),
\begin{equation}
\label{Jsplit1}
    2 J_T(\sigma,\theta)
     =
    \inf
        \int_{0}^{T}
        (
            |\beta_t|^2
            +
            \Tr(K_t \Pi_t K_t^{\rT})
        )
        \rd t
    =
    S_T(\Pi_0, \Pi_T)
    +
    \inf
    \int_{0}^{T}
    |\beta_t|^2\rd t.
\end{equation}
Here, the first infimum is over the functions $\beta_t$
and $K_t$ from (\ref{betaK}) such that the state mean $\alpha_t$ and
covariance matrix $\Pi_t$, governed by the ODEs (\ref{alphaPidot}) with initial conditions $\alpha_0$ and $\Pi_0$,
satisfy the terminal conditions specified by $\alpha_T$ and $\Pi_T$, whereas the second infimum
is only concerned with $\beta_t$.
By the first of the ODEs (\ref{alphaPidot}), the boundary
conditions on the state mean are equivalent to
$
    \int_{0}^{T}\re^{A(T-t)} B \beta_t
    \rd t
    =
    \alpha_T
    -
    M_T(\alpha_0),
$
where (\ref{M}) is used.
The second infimum in (\ref{Jsplit1}) is found by solving the linearly constrained
quadratic optimization problem and achieved at the function $\beta_t$ from (\ref{betaoptimal}), with
$    \min
        \int_{0}^{T}
            |\beta_t|^2
        \rd t
=
    \|
        \alpha_T
        -
        M_T(\alpha_0)
    \|_{\Gamma_T^{-1}}^2
$, which yields
(\ref{Jsplit}). \end{proof}

The representation (\ref{Jsplit}) splits $J_T$ into two
independently computed  terms which are associated with the mean and
the covariance matrix of the system state. This decoupling is
similar to the separation principle of Linear Quadratic Gaussian  control \cite{KS_1972}.

\section{Covariance Hamilton-Jacobi equation\label{sec:covHJ}}

We will now consider the ``covariance'' part $S_T$ of the minimum
required supply $J_T$ from (\ref{Jsplit}) whose ``mean'' part is already computed in  Theorem~\ref{th:split}.

\begin{lemma}
\label{lem:HJ_S} Suppose that $S_T(\Sigma, \Theta)$ from
(\ref{S}) is smooth with respect to
$T>0$ and $\Sigma\succ 0$.
Then it satisfies the HJE
\begin{equation}
\label{HJ_S}
    \d_T S_T
    =
    F
    (
        \Sigma,
        \d_{\Sigma} S_T
    ),\qquad
    F(\Sigma, \Phi)
    :=
    \Tr
    (
        (
            \cC(\Sigma)
            -
            D
            \Phi
            \Sigma
        )
        \Phi
    ),
\end{equation}
where (\ref{DBB}) and
 (\ref{MCdot}) are used.
\end{lemma}
\begin{proof} From (\ref{alphaPidot}) and (\ref{S}), it follows that,  under the assumption of smoothness, $S_T$ satisfies a Hamilton-Jacobi-Bellman equation
\begin{equation}
\label{HJB_S}
    \d_T S_T
     =
    \Tr
    (
        \cC(\Sigma)
        \Phi_T
    )
+
    \min_{K \in \mR^{m\x n}}
    (
        \Tr(K\Sigma K^{\rT})
        +
         \Tr(
            (
                BK\Sigma + \Sigma K^{\rT} B^{\rT}
            )
            \Phi_T
        )
    ),
\end{equation}
where     $\Phi_T
    :=
    \d_{\Sigma}S_T$. Since $
    K\Sigma K^{\rT}
    +
    K \Sigma \Phi B + B^{\rT} \Phi \Sigma K^{\rT}
     =
    (K+B^{\rT}\Phi)\Sigma (K^{\rT}+\Phi B)-B^{\rT} \Phi\Sigma \Phi
    B
$ and $\Sigma\succ 0$, the minimum in (\ref{HJB_S})   is only achieved at
$    K = -B^{\rT} \Phi_T
$ and is equal to $    -\Tr
    (
        D\Phi_T\Sigma \Phi_T
    )
$,
which implies
 (\ref{HJ_S}).
\end{proof}

\section{Correction scheme \label{sec:correction_scheme}}

The Hamiltonian $F$ in
 (\ref{HJ_S})  is quadratic in its second
 argument. Hence, if $\wh{S}_T$ satisfies the covariance HJE, then
\begin{equation}
\label{SSS}
    S_T
    :=
    \wh{S}_T + \wt{S}_T
\end{equation}
is also a solution of (\ref{HJ_S}) if and only if  $\wt{S}_T$,
playing the role  of a correcting function, satisfies a modified HJE
\begin{equation}
\label{Finter1}
    \d_T \wt{S}_T
     =
    \wh{F}
    (
        \Sigma,
        \d_{\Sigma}\wt{S}_T
    ).
\end{equation}
Here,  $\wh{F}$
is obtained by correcting the Hamiltonian $F$ by a term arising from the bilinear
``interaction'' between $\wh{S}_T$ and $\wt{S}_T$:
\begin{equation}
\label{Finter2}
    \wh{F}
    (
        \Sigma,
        \Phi
    )
    :=
    F
    (
        \Sigma,
        \Phi
    )
    -
    2
    \Tr
    (
        D
        (\d_{\Sigma}\wh{S}_T)
        \Sigma
        \Phi
    )
    =
    \Tr
    (
        (
            \wh{\cC}(\Sigma)
            -
            D
            \Phi\Sigma
        )
        \Phi
    ),
\end{equation}
with
\begin{equation}
\label{Finter3}
    \wh{\cC}(\Sigma)
     :=
     (
        A-D\d_{\Sigma}\wh{S}_T
     )
     \Sigma
     +
     \Sigma
     (
        A-D\d_{\Sigma}\wh{S}_T
     )^{\rT}
     +D
\end{equation}
obtained by modifying the matrix $A$ in the infinitesimal generator
$\cC$ from (\ref{MCdot}). The operator $\wh{\cC}$ depends
parametrically  on $T$ and $\Sigma$ through $\wh{S}_T$.

\section{Starting solution \label{sec:first_ansatz}}

Ignoring, for the moment, the boundary condition (\ref{S0}), we will
find a particular solution $\wh{S}_T$ of the covariance HJE
(\ref{HJ_S}) as a starting point for the correction scheme
(\ref{SSS})--(\ref{Finter3}). Note that application of
the Hamiltonian $F$ from (\ref{HJ_S}) to the function $\Phi:= \d_{\Sigma}\ln\det
\Sigma = \Sigma^{-1}$ yields a constant:
$
    F(\Sigma, \Sigma^{-1})
    =
    \Tr
    (
        (
            A\Sigma
            +
            \Sigma A^{\rT} + D - D\Sigma^{-1} \Sigma
        )
        \Sigma^{-1}
    )
    =
    2\Tr A
$. Furthermore, the class of those smooth functions $S_T(\Sigma,
\Theta)$, which are affine with respect to $\Sigma$, is closed under
$\d_T$ and the action of the Hamiltonian $S_T \mapsto F(\Sigma,
\d_{\Sigma}S_T)$. Finally, solutions of the quadratic HJE
(\ref{HJ_S}) do not obey the superposition principle, and the
bilinear interaction of $\ln\det \Sigma$ and an affine function
$S_T$ is described by
$
    \Tr
    (
        D\Sigma^{-1}
        \Sigma
        \d_{\Sigma}S_T
    )
    =
    \Tr
    (
        D
        \d_{\Sigma}
        S_T
    )
$
which is also constant in $\Sigma$. These  observations suggest
looking for a particular solution of (\ref{HJ_S}) in the class of
functions
\begin{equation}
\label{classShat}
    \wh{S}_T(\Sigma, \Theta)
    =
    \ln\det \Sigma
    +
    \Tr
    (
        \Sigma
        \Xi_T(\Theta)
    )
    +
    \Ups_T(\Theta).
\end{equation}
Here, $\Xi_T(\Theta)$ and $\Ups_T(\Theta)$ are smooth functions of
the time horizon $T> 0$ and the terminal state covariance matrix
$\Theta$  with values in $\mS_n$ and $\mR$, respectively, with  $\det\Xi_T(\Theta)\ne 0$.

\begin{lemma}
\label{lem:Shat} The following function is a particular solution of
the covariance HJE (\ref{HJ_S}) in the class (\ref{classShat}):
\begin{equation}
\label{Shat}
    \wh{S}_T(\Sigma)
    :=
    \ln\det U_T(\Sigma) + \Tr U_T(\Sigma).
\end{equation}
Here, the map $U_T: \mP_n \to \mP_n$ is defined using  (\ref{Gamma}), (\ref{C}) as
\begin{equation}
\label{U}
    U_T(\Sigma)
      :=
    \Gamma_T^{-1/2}
    C_T(\Sigma)
    \Gamma_T^{-1/2}
    -
    I_n.
\end{equation}
\end{lemma}
\begin{proof} By differentiating the ansatz (\ref{classShat}), it follows that
\begin{equation}
\label{Shatdiff}
    \d_T\wh{S}_T
     =
    \Tr
    (
        \Sigma
        \d_T\Xi_T
    )
        +
        \d_T\Ups_T,
        \qquad
    \d_{\Sigma}\wh{S}_T
    =
    \Sigma^{-1} + \Xi_T.
\end{equation}
In view of (\ref{MCdot}), substitution of $\d_{\Sigma}\wh{S}_T$ into
(\ref{HJ_S}) yields
\begin{align}
\nonumber
    F
    (
        \Sigma,
        \d_T\wh{S}_T
    )
\nonumber
     &=
    \Tr
    (
        (
            A\Sigma + \Sigma A^{\rT}
            + D
            -
            D
            (
                \Sigma^{-1}
                +
                \Xi_T
            )
            \Sigma
        )
        \d_{\Sigma}\wh{S}_T
    )\\
\nonumber
     &=
    \Tr
    (
        (
            A\Sigma + \Sigma A^{\rT}
            -
            D
                \Xi_T
            \Sigma
        )
        (
            \Sigma^{-1}
            +
            \Xi_T
        )
    )\\
\label{Shatright}
     &=
    \Tr
    (
        (
            \Xi_T A + A^{\rT} \Xi_T - \Xi_T D\Xi_T
        )
        \Sigma
    )
     +
    \Tr
    (
        2A - D\Xi_T
    ).
\end{align}
By equating  $\d_T\wh{S}_T$  from (\ref{Shatdiff})
with the right-hand side of (\ref{Shatright}), it follows that the
ansatz function $\wh{S}_T$ in (\ref{classShat}) is a solution of
(\ref{HJ_S}) if and only if $\Xi_T$ and $\Ups_T$ satisfy the ODEs
\begin{align}
\label{Xi_diff}
    \d_T\Xi_T
    & =
    \Xi_T A + A^{\rT} \Xi_T - \Xi_T D \Xi_T,\\
\label{Ups_diff}
    \d_T\Ups_T
    & =
    \Tr (2A-D\Xi_T).
\end{align}
Multiplication of the right-hand side of (\ref{Xi_diff})
by $\Xi_T^{-1}$ yields a matrix whose trace coincides with the
right-hand side of (\ref{Ups_diff}). Hence,
$    \d_T\ln |\det \Xi_T|
    =
    \Tr
    (
        \Xi_T^{-1}
        \d_T\Xi_T
    )
    =
    \d_T\Ups_T
$, and
\begin{equation}
\label{Ups}
    \Ups_T
    =
    \ln|\det \Xi_T|
    +
    \Ups_0
    -
    \ln|\det \Xi_0|.
\end{equation}
Left and right multiplication of both parts of  (\ref{Xi_diff}) by $\Xi_T^{-1}$ yields
a differential Lyapunov equation
\begin{equation}
\label{Xidiff}
    \d_T(\Xi_T^{-1})
    =
    -
    \Xi_T^{-1}
    (\d_T\Xi_T)
    \Xi_T^{-1}
    =
    D-A\Xi_T^{-1}
    -
    \Xi_T^{-1}A^{\rT}
\end{equation}
for  $\Xi_T^{-1}$ with a unique equilibrium point
$\Xi_T = -\Pi_*^{-1}$; cf. (\ref{Pi*}).
The general solution of (\ref{Xidiff}) is expressed via
 (\ref{Gamma}) as
\begin{equation}
\label{Xi}
    \Xi_T
    =
    \re^{A^{\rT}T}
    (
        \Xi_0^{-1} + \Gamma_T
    )^{-1}
    \re^{AT}.
\end{equation}
Setting $\Ups_0 := \ln|\det \Xi_0|$ in (\ref{Ups}) and
$\Xi_0 \to \infty$ in (\ref{Xi}) gives
\begin{equation}
\label{hatXiUps}
\Xi_T
     =
    \re^{A^{\rT} T}
    \Gamma_T^{-1}
    \re^{AT},
    \qquad
\Ups_T
     =
    2 T \Tr A
    -
    \ln\det \Gamma_T.
\end{equation}
Substitution of (\ref{hatXiUps}) into (\ref{classShat}) yields the
particular solution of (\ref{HJ_S}) described by
(\ref{Shat})--(\ref{U}).
\end{proof}

\section{Trace-analytic correction\label{sec:trace_analytic_correction}}

Since $\wh{S}_T\to -\infty$ as $\Sigma\to 0$, regardless of the choice of
 $\Xi_T$ and $\Ups_T$, the class   (\ref{classShat})  can not provide a nonnegative solution to the HJE  (\ref{HJ_S}). We will therefore correct
$\wh{S}_T$ from (\ref{Shat}) by adding a function $\wt{S}_T :=
\wt{S}_T(\Sigma,\Theta)$ such that (\ref{SSS}) is a nonnegative
solution of (\ref{HJ_S}) satisfying the boundary condition
(\ref{S0}). Substitution of $\d_{\Sigma}\wh{S}_T$ from
(\ref{Shatdiff}) into the correction scheme
(\ref{Finter1})--(\ref{Finter3}) yields
\begin{equation}
\label{HJ_Stilde}
    \d_T\wt{S}_T
     =
    \Tr
    (
        (\,
            \overbrace{
                (A-D\Xi_T)\Sigma + \Sigma (A-D\Xi_T)^{\rT} -
                D
            }^{\wh{\cC}(\Sigma)}
            -
            D
            (\d_{\Sigma} \wt{S}_T)
            \Sigma
        )
        \d_{\Sigma} \wt{S}_T
    ).
\end{equation}
To find the correcting function $\wt{S}_T$,
 as a solution to
(\ref{HJ_Stilde}) such that $S_T$ in (\ref{SSS}) is nonnegative and
satisfies (\ref{S0}),
we will employ yet
another ansatz:
\begin{equation}
\label{classStilde}
    \wt{S}_T
    :=
    \Tr \chi(\Omega_T)
    +
    \rho_T,
    \qquad
    \Omega_T
    :=
    \Psi_T
    \Sigma
    \Psi_T.
\end{equation}
Here, $\chi$ is a nonconstant function of a complex variable,
analytic in a neighbourhood of $\mR_+$ and real-valued on $\mR_+$.
Also, $\Psi_T(\Theta)$ and $\rho_T(\Theta)$ are smooth functions of
the time horizon $T$ and the terminal state covariance matrix
$\Theta$ with values in $\mP_n$ and $\mR$, respectively.
The matrix $\Psi_T$, which  specifies the linear operator $\Sigma \mapsto \Omega_T$ in (\ref{classStilde}), enters the ``trace-analytic'' function $\Tr\chi(\Omega_T)$ only
through $\Psi_T^2$, since
\begin{equation}
\label{chi_Psi}
    \Tr \chi(\Omega_T)
    =
    \Tr \chi(\Psi_T^{-1} \Psi_T^2\Sigma  \Psi_T )
    =
    \Tr\chi(\Psi_T^2 \Sigma),
\end{equation}
where we use the invariance of the trace under
similarity transformations and their commutativity with analytic
functions of matrices.  To find suitable $\chi$, $\Psi_T$,
$\rho_T$, we compute the derivatives
$
    \dot{\wt{S}}
     =
    \Tr
    (
        (
            \dot{\Psi}\Psi^{-1}
            +
            \Psi^{-1} \dot{\Psi}
        )
        \chi\,'(\Omega) \Omega
    )
    +
    \dot{\rho}
 $ and $
    \d_{\Sigma}\wt{S}
     =
    \Psi\chi\,'(\Omega)\Psi
$
by using Lemma~\ref{lem:diff} and its corollary
(\ref{phidiffdiff})--(\ref{chidiffdiff})  in Appendix~A, with
the subscript $T$ omitted for brevity, and $\dot{(\,)} := \d_T$. Substituting the
 derivatives into (\ref{HJ_Stilde}) yields
\begin{align}
\nonumber
    \Tr
    (
        (
            \dot{\Psi}\Psi^{-1}
            +
            \Psi^{-1} \dot{\Psi}
        )
        \chi\,'(\Omega) \Omega
    )
    +
    \dot{\rho}
    =&
    \Tr
    (
        (
            \wh{\cC}(\Sigma)
            -D
            \Psi \chi\,'(\Omega) \Psi
            \Sigma
        )
        \Psi \chi\,'(\Omega) \Psi
    )\\
\nonumber
    =&
    \Tr
    (
        (
            \Psi
            (A-D\Xi)
            \Psi^{-1}
            +
            \Psi^{-1}
            (A-D\Xi)^{\rT}
            \Psi
        )
        \chi\,'(\Omega)\Omega
    )\\
\label{HJchi}
     & -
    \Tr
    (
        \Psi D \Psi
        \chi\,'(\Omega)
        (
            I_n + \chi\,'(\Omega)\Omega
        )
    ),
\end{align}
where use is made of $\Omega = \Psi \Sigma \Psi$ from
(\ref{classStilde}) and the commutativity of
$\chi\,'(\Omega)$ and $\Omega$. Regrouping the terms of (\ref{HJchi}) yields
\begin{equation}
\label{GOmega}
    \sum_{k=1}^{2}
    \Tr(G_k\chi_k(\Omega_T))
    +
    \d_T\rho_T = 0
\end{equation}
for all $T>0$ and $\Sigma, \Theta\succ 0$. Here,
\begin{align}
\label{G1}
    G_1(T,\Theta)
      &:= 
    \Psi_T D\Psi_T,\\
\label{G2}
    G_2(T,\Theta)
     &:=
            (\d_T\Psi_T)\Psi_T^{-1}
            +
            \Psi_T^{-1} \d_T\Psi_T
              -
            \Psi_T
            (A-D\Xi_T)
            \Psi_T^{-1}
            -
            \Psi_T^{-1}
            (A-D\Xi_T)^{\rT}
            \Psi_T
\end{align}
are $\mS_n$-valued functions, with $G_1(T,\Theta)\succ
0$. Also,
\begin{equation}
\label{chi12}
    \chi_1(\omega)
    :=
    \chi\,'(\omega)
    (1+\omega\chi\,'(\omega)),
    \qquad
    \chi_2(\omega)
    :=
    \omega
    \chi\,'(\omega)
\end{equation}
are functions of a complex variable $\omega$, which inherit from
$\chi$ the analyticity in a neighbourhood of and
real-valuedness on $\mR_+$. Since
$\chi_1(\omega)=\chi_2(\omega)(1+\chi_2(\omega))/\omega$, then  $\chi_1$ or $\chi_2$ is not constant.
For any given $(T,\Theta)$, the map $\Sigma\mapsto \Omega_T$
in (\ref{classStilde}) is a bijection of $\mP_n$, and hence, (\ref{GOmega}) is equivalent to
\begin{equation}
\label{GOmega1}
    \sum_{k=1}^{2}
    \Tr(G_k(T,\Theta)\chi_k(\Omega))
    +
    \d_T\rho_T = 0
\end{equation}
for all
    $T>0$ and $\Theta, \Omega\succ 0$.
Application of the separation-of-variables principle of
Lemma~\ref{lem:matrix_separation} from Appendix~B to
(\ref{GOmega1}) yields the existence of constants
$\lambda, \tau\in \mR$ such that the function $\chi$, which generates
$\chi_1$, $\chi_2$  in (\ref{chi12}), satisfies the ODE
\begin{equation}
\label{chi_ODE}
    \chi\,'(\omega)
    (1+\omega\chi\,'(\omega))
    +
    \lambda \omega\chi\,'(\omega) + \tau = 0,
\end{equation}
and $G_1$ and $G_2$ in (\ref{G1}) and (\ref{G2}) and
$\rho_T$ from (\ref{classStilde}) satisfy
\begin{align}
\label{Psi_PDE}
            (\d_T\Psi_T)\Psi_T^{-1}
              +
            \Psi_T^{-1} \d_T\Psi_T
               -
            \Psi_T
            (A-D\Xi_T)
            \Psi_T^{-1}
            -
            \Psi_T^{-1}
            (A-D\Xi_T)^{\rT}
            \Psi_T
             & =
            \lambda
            \Psi_T D \Psi_T,\\
\label{rho_PDE}
        \d_T\rho_T
         & =
        \tau \Tr(D \Psi_T^2).
\end{align}
The PDEs (\ref{Psi_PDE}) and (\ref{rho_PDE}) are solved in the
lemma below and the result is then combined with the ODE
(\ref{chi_ODE}).

\begin{lemma}
\label{lem:Stilde} The following function is a solution of the
modified HJE  (\ref{HJ_Stilde}) in the class (\ref{classStilde}):
\begin{equation}
\label{Stilde_solution}
    \wt{S}_T(\Sigma, \Theta)
    =
    \Tr
    \chi(
        U_T(\Sigma)
        V_T(\Theta)
    )
    +
    \rho_T(\Theta),
\end{equation}
where $U_T$ is given by (\ref{U}) and $\chi$ satisfies
the ODE (\ref{chi_ODE}). Here, for any $\lambda \> 0$, the map $V_T:
\mP_n\to \mP_n$ is associated with  (\ref{Gamma}) by
\begin{equation}
\label{V_mho}
    V_T(\Theta)
    :=
    \Gamma_T^{-1/2}
    \mho_T(\Theta)^{-1}
    \Gamma_T^{-1/2},
    \
    \mho_T(\Theta)
    :=
    \lambda \Gamma_T^{-1} + \mho(\Theta),
 \end{equation}
where $\mho(\Theta)$ is a $\mS_n$-valued function of $\Theta$ only which satisfies
\begin{equation}
\label{mho_bottom}
    \mho(\Theta)
    \succcurlyeq
    -\lambda
    \Pi_*^{-1}\ {\rm for}\ \lambda > 0,
    \qquad
    \mho(\Theta)
    \succ 0\
    {\rm for}\ \lambda = 0,
\end{equation}
with $\Pi_*$  given by
(\ref{alphaPi*}). Also,
\begin{equation}
\label{rho_solution}
    \rho_T(\Theta)
    =
    \rho(\Theta)
    -
    \left\{
    \begin{array}{cl}
    (\tau/\lambda)
    \ln
    \det
    \mho_T(\Theta) & {\rm for}\ \lambda >0\\
    \tau\Tr(\mho(\Theta)^{-1}\Gamma_T^{-1}) & {\rm for}\ \lambda =0
    \end{array}
    \right.,
\end{equation}
where $\rho(\Theta)$ is a $\mR$-valued function of $\Theta$ only.
\end{lemma}
\begin{proof} Since $
    \Psi^{-1}
        \dot{\Psi}
    \Psi^{-2}$
    + $
    \Psi^{-2}
        \dot{\Psi}
    \Psi^{-1}$
    = $
    \Psi^{-2}
    (\Psi^2)^{\centerdot}
    \Psi^{-2}
    $ = $
    -
    (\Psi^{-2})^{\centerdot}
$,  left and right multiplication of both sides of (\ref{Psi_PDE})
by $\Psi_T^{-1}$ yields a differential Lyapunov equation
\begin{equation}
\label{Psi_PDE1}
            \d_T(\Psi_T^{-2})
            +
            (A-D\Xi_T)
            \Psi_T^{-2}
            +
            \Psi_T^{-2}
            (A-D\Xi_T)^{\rT}
            +
            \lambda
            D
            =0
\end{equation}
with respect to $\Psi_T^{-2}$.
Its solution is expressed in terms of the fundamental
solution of the ODE
\begin{equation}
\label{aux_ODE}
     \d_T\xi_T
     =
    -(A-D\Xi_T)\xi_T
    =
        -(A-D    \re^{A^{\rT} T}
    \Gamma_T^{-1}
    \re^{AT}
    )
    \xi_T,
\end{equation}
where $\Xi_T$ is given by (\ref{hatXiUps}). The relation $
    \dot{\Gamma}_t
    =
    \re^{At}D\re^{A^{\rT}t}
$ implies that the general solution of (\ref{aux_ODE}) is $
    \xi_T = \re^{-AT} \Gamma_T \zeta
$,  where $\zeta\in \mR^n$ is an arbitrary constant vector. Hence,
the solution of (\ref{Psi_PDE1}) can be found in the form
\begin{equation}\label{Psi_form}
    \Psi_T^{-2}
    =
    \re^{-AT}
    \Gamma_T
    \mho_T
    \Gamma_T
    \re^{-A^{\rT}T}
\end{equation}
by the variation of constants method, where $\mho_T(\Theta)$ is a
$\mS_n$-valued function of $T,\Theta$. Substitution of
(\ref{Psi_form}) into (\ref{Psi_PDE1}) yields
\begin{equation}
\label{mho_PDE}
    \d_T \mho_T
    =
    -\lambda
    \Gamma_T^{-1}
    \re^{AT}
    D
    \re^{A^{\rT}T}
    \Gamma_T^{-1}
    =
    \lambda
    \d_T (\Gamma_T^{-1}),
\end{equation}
whence $\mho_T$ in (\ref{V_mho}) is obtained by integration. Now,
the right-hand side of (\ref{Psi_form}) is positive definite for all
$T>0$ if and only if so is $\mho_T(\Theta)$. In view of
(\ref{V_mho}), the condition $\mho_T(\Theta)\succ 0$ for any $T>0$ is
equivalent to $\lambda\>0$ and (\ref{mho_bottom}).
Indeed, 1) the controllability Gramian in
(\ref{Gamma}) satisfies $0\prec \Gamma_T\prec \Pi_*$ for all $T>0$,
and strictly $\prec$-monotonically  approaches
 $0$ and $\Pi_*$   as $T$ tends to $0$ and
$+\infty$, respectively; 2) $(\cdot)^{-1}$ is a
decreasing operator on the set $\mP_n$ with respect to the partial
ordering $\prec$. Therefore,
\begin{equation}
\label{Psi_solution}
    \Psi_T(\Theta)^2
     =
    \re^{A^{\rT}T}
    \Gamma_T^{-1}
    \mho_T(\Theta)^{-1}
    \Gamma_T^{-1}
    \re^{AT},
\end{equation}
which is obtained as the matrix
inverse of (\ref{Psi_form}). Substitution of (\ref{Psi_solution}) into (\ref{chi_Psi})
yields
\begin{align}
\nonumber
    \Tr
    \chi(\Omega_T)
    & =
    \Tr
    \chi
    (
        \Sigma
        \re^{A^{\rT}T}
        \Gamma_T^{-1}
        \mho_T(\Theta)^{-1}
        \Gamma_T^{-1}
        \re^{AT}
    )\\
\label{chi_UV}
    & =
    \Tr
    \chi
    (
        \Gamma_T^{-1/2}
        \re^{AT}
        \Sigma
        \re^{A^{\rT}T}
        \Gamma_T^{-1/2}
        \Gamma_T^{-1/2}
        \mho_T(\Theta)^{-1}
        \Gamma_T^{-1/2}
    )
     =
    \Tr
    \chi
    (
        U_T(\Sigma)
        V_T(\Theta)
    ),
\end{align}
where we have used (\ref{U}) and $V_T$ from (\ref{V_mho}). By combining
(\ref{Psi_solution}) with (\ref{rho_PDE}) and (\ref{mho_PDE}), it follows that
\begin{align}
\nonumber
        \d_T\rho_T
         & =
        \tau
        \Tr
        (
            D
            \re^{A^{\rT}T}
            \Gamma_T^{-1}
            \mho_T(\Theta)^{-1}
            \Gamma_T^{-1}
            \re^{AT}
        )\\
\label{rho_PDE1}
        & =
        \tau
        \Tr
        (
            \mho_T(\Theta)^{-1}
            \Gamma_T^{-1}
            \re^{AT}
            D
            \re^{A^{\rT}T}
            \Gamma_T^{-1}
        )
    =
        -
        \left\{
        \begin{array}{cl}
            (\tau/\lambda)\d_T\ln\det\mho_T(\Theta)
            & {\rm for}\ \lambda > 0\\
            \tau\d_T\Tr(\mho(\Theta)^{-1}\Gamma_T^{-1})
            & {\rm for}\ \lambda = 0
        \end{array}
        \right..
\end{align}
Here, use is made of  the property that, in the case $\lambda =
0$, the matrix  $\mho_T(\Theta)$ from (\ref{V_mho}) does not depend
on $T$. Integration of the
PDE (\ref{rho_PDE1}) yields (\ref{rho_solution}).
Assembling the latter and
(\ref{chi_UV}) into (\ref{classStilde}) leads to (\ref{Stilde_solution}).
\end{proof}

For any constant $s>0$, the function $\wt{S}_T$, described by
Lemma~\ref{lem:Stilde}, is invariant under the scaling
transformation
\begin{equation}
\label{scaling}
    \chi(\omega)\mapsto \chi(s\omega),
    \qquad
    \mho\mapsto s \mho,
    \qquad
    \lambda\mapsto s \lambda,
    \qquad
    \tau\mapsto s\tau,
\end{equation}
and so are the ratio $\lambda/\tau$ and the ODE (\ref{chi_ODE}). We
now combine Lemmas~\ref{lem:Shat} and \ref{lem:Stilde} to
finalise the correction scheme (\ref{SSS})--(\ref{Finter3}).

\begin{theorem}
\label{th:SSS} A nonnegative solution of the covariance HJE  (\ref{HJ_S}),
satisfying the boundary condition (\ref{S0}), is given by
\begin{equation}
\label{SS}
    S_T(\Sigma,\Theta)
    :=
    \ln\det (2U_T(\Sigma))
    +
    \Tr
    (
        U_T(\Sigma) + V_T(\Theta)
        +
        \chi(
            U_T(\Sigma)
            V_T(\Theta)
        )
    ).
\end{equation}
Here,
\begin{equation}
\label{bingo}
    \chi(\omega)
    =
    -\sqrt{1+4\omega}
    -
    \ln
    (
        \sqrt{1+4\omega} - 1
    ),
\end{equation}
the map $U_T$ is given by (\ref{U}), and $V_T$ in
(\ref{V_mho}) takes the form
\begin{equation}
\label{VV}
    V_T(\Theta)
    =
    \Gamma_T^{-1/2} \Theta \Gamma_T^{-1/2}.
\end{equation}
\end{theorem}
\begin{proof} By solving (\ref{chi_ODE}) as a quadratic
equation with respect to  $\chi\,'(\omega)$, it follows that
\begin{equation}
\label{chi_dash}
    \chi\,'(\omega)
    =
    (
        -(\lambda \omega + 1)
        \pm
        \sqrt{(\lambda \omega + 1)^2 - 4\tau\omega}
    )
    /
    (2\omega).
\end{equation}
The function $\chi$  can now be obtained by integrating one of the two
regular branches of (\ref{chi_dash}) in $\omega$. This integration is straightforward in
the case $\lambda := 0$ which is shown below to yield a
feasible solution in (\ref{SS}). In this case, $\tau$ must be negative to make $\chi(\omega)$ real for all $\omega \> 0$ (the
trivial situation $\lambda = \tau = 0$ is excluded from
consideration). In view of
(\ref{scaling}), we set $\tau := -1$ without loss of generality.
Integration of
$
    \chi\,'(\omega)
    =
    -
    (1 + \sqrt{1+4\omega})/(2\omega)
$
yields (\ref{bingo}).  This particular choice of the branch is
motivated by the identity
\begin{equation}
\label{chi_identity}
    \chi((v-1)v) = 1-2v-\ln(2 (v-1))
\end{equation}
for $\chi$ from
(\ref{bingo}),  which is crucial in what follows to achieve the
fulfillment of (\ref{S0}). To this end, by
assembling the starting solution $\wh{S}_T$ from
Lemma~\ref{lem:Shat} and the trace-analytic correcting function
$\wt{S}_T$ from Lemma~\ref{lem:Stilde} (with $
    \lambda := 0 $,
$
    \tau := -1
$)
 into
(\ref{SSS}), it follows that
\begin{equation}
\label{Strial}
    S_T(\Sigma,\Theta)
     =
    \ln\det U_T(\Sigma) + \Tr U_T(\Sigma)
    +
    \Tr
    \chi(
        U_T(\Sigma)
        V_T(\Theta)
    )
    +
    \rho(\Theta) + \Tr V_T(\Theta),
\end{equation}
where the map $V_T$ is given by (\ref{V_mho}) with
$\lambda = 0$, so that
\begin{equation}\label{VVV}
    V_T(\Theta)
    =
    \Gamma_T^{-1/2}\mho(\Theta)^{-1}\Gamma_T^{-1/2}.
\end{equation}
 From (\ref{U}), it follows that if the terminal state covariance matrix $\Theta$ is nominally
reachable from the initial state covariance matrix $\Sigma$ in time
$T$, then $
    U_T(\Sigma)
    =
    \Gamma_T^{-1/2} \Theta \Gamma_T^{-1/2} - I_n
$ for $\Theta = C_T(\Sigma)$. Hence, $S_T$ in (\ref{Strial}) satisfies
(\ref{S0}) if and only if
\begin{align}
\nonumber
    S_T(C_T^{-1}(\Theta),\Theta)
    = &
    \ln\det
    (\Theta \Gamma_T^{-1}-I_n)
    +
    \Tr (\Theta \Gamma_T^{-1}-I_n)\\
\label{Strial0}
    & +
    \Tr
    \chi(
        (\Gamma_T^{-1/2} \Theta \Gamma_T^{-1/2}-I_n)
        V_T(\Theta)
    )
    +
    \rho(\Theta) + \Tr V_T(\Theta)
    =0
\end{align}
for all $
    \Theta \succ \Gamma_T$.
With $\chi$ given by (\ref{bingo}), it is
possible to find $\mho$, $\rho$ so as to satisfy
(\ref{Strial0}). Indeed, by setting $
    \mho(\Theta):= \Theta^{-1}$ and $
    \rho(\Theta) := n\ln 2
$,  the representation (\ref{VVV}) becomes (\ref{VV})  and the
boundary value of $S_T$ from (\ref{Strial0}) takes the
form
\begin{align*}
    S_T(C_T^{-1}(\Theta),\Theta)
    =&
    \ln\det
    (V_T(\Theta)-I_n)
    +
    \Tr (V_T(\Theta)-I_n)\\
    &+
    \Tr
    \chi(
        (V_T(\Theta)-I_n)
        V_T(\Theta)
    )
    +
    n\ln 2 + \Tr V_T(\Theta)\\
    =&
    \ln\det
    (2(V_T-I_n))
    +
    \Tr(2V_T-I_n)
    +
    \Tr
    \chi(
        (V_T-I_n)
        V_T
    ).
\end{align*}
The right-hand side of this equation
vanishes for all nominally reachable  $\Theta \succ \Gamma_T$  in view of
(\ref{chi_identity}). \end{proof}

The function $S_T$ from Theorem~\ref{th:SSS} is a
smooth solution of the covariance HJE (\ref{HJ_S}).
The minimum required
supply in the state PDF transition problem (\ref{J}) is obtained by combining Theorems~\ref{th:split} and
\ref{th:SSS}:
\begin{align}
\nonumber
    J_T(
        \cN(\alpha_0, \Pi_0),
        \cN(\alpha_T, \Pi_T)
    )
     =&
    (
        \|\alpha_T - M_T(\alpha_0)\|_{\Gamma_T^{-1}}^2
+
    \Tr
    (
        U_T + V_T
        -
        \sqrt{I_n + 4U_T V_T}
    )\\
\label{Jfinal}
    &-
        \ln\det
        (
            (
                \sqrt{I_n + 4U_T V_T} - I_n
            )
            (2U_T)^{-1}
        )
    )/2,
\end{align}
with $
    U_T
    =
    \Gamma_T^{-1/2}C_T(\Pi_0)\Gamma_T^{-1/2}-I_n$,
$    V_T
    =
    \Gamma_T^{-1/2} \Pi_T \Gamma_T^{-1/2}
$. Here, $M_T$, $C_T$ are the nominal state mean and covariance
semigroups from (\ref{M}), (\ref{C}), and $\Gamma_T$ is the
finite-horizon controllability Gramian from (\ref{Gamma}).
The structure of the right-hand side of
(\ref{Jfinal}) is identical to that in the discrete time case
\cite{VP_2009}, except that the semigroups and the gramian are
computed in accordance with the continuous time setting.


\appendix
\section*{A. Differentiation of trace-analytic functions}
\renewcommand{\theequation}{A\arabic{equation}}
\setcounter{equation}{0}

The following lemma computes the Frechet derivative for
the composition of an
analytic function with the matrix trace.
\begin{lemma}
\label{lem:diff} Let $\chi$ be a function of a complex variable,
analytic in a neighbourhood of $\mR_+$. Then the Frechet derivative
of the function $\mP_n \ni \Sigma \mapsto \Tr \chi(\Sigma)\in \mC$
is
\begin{equation}
\label{phidiff}
    \d_{\Sigma}\Tr \chi(\Sigma)
    =
    \chi\,'(\Sigma).
\end{equation}
\end{lemma}
\begin{proof} Applying an elementary polynomial $
    \chi(z)
    :=
    z^k
$  of degree $k\> 1$ to a matrix $\Sigma \in \mS_n$ yields the first variation
$
    \delta \Tr (\Sigma^k)
     =
    \sum_{s=0}^{k-1}
    \Tr
    (
        \Sigma^s
        (\delta \Sigma)
        \Sigma^{k-1-s}
    )
     =
     k
    \Tr
    (
        \Sigma^{k-1}
        \delta \Sigma
    )
$;
cf. \cite[p.~270]{SIG_1998}. By linearity, this proves (\ref{phidiff})
for arbitrary
polynomials $\chi$. A standard passage to the limit extends
(\ref{phidiff}) to a power series
$\chi(z):=\sum_{k=0}^{+\infty}c_k(z-z_0)^k$ whose disk of
convergence  contains the
spectrum of $\Sigma$.
\end{proof}

For the purposes of Section~\ref{sec:trace_analytic_correction},
we
complement Lemma~\ref{lem:diff} by two differentiation formulae. 
 let $\Psi$ be a smooth $\mP_n$-valued function of the
independent time variable $T$ which generates  a $\mP_n$-valued
function $\Omega:= \Psi \Sigma \Psi$ of $T$ and $\Sigma\in \mP_n$.
 Then for any $\chi$ from Lemma~\ref{lem:diff},
\begin{align}
\label{phidiffdiff}
    (\Tr \chi(\Omega))^{\centerdot}
    & =
    \Tr
    (
        (
            \dot{\Psi}\Psi^{-1}
            +
            \Psi^{-1} \dot{\Psi}
        )
        \chi\,'(\Omega)\Omega
    ),\\
\label{chidiffdiff}
    \d_{\Sigma}\Tr \chi(\Omega)
    & =
    \Psi\chi\,'(\Omega)\Psi,
\end{align}
where $\dot{(\, )}:=\d_T$. Indeed, (\ref{phidiffdiff}) is obtained by applying the chain
rule to the composition of $\Tr\chi$ and $\Omega$ as a function of
$T$ for a fixed $\Sigma$ and employing (\ref{phidiff}):
$
    (\Tr\chi(\Omega))^{\centerdot}
     =
    \Tr
    (
        \chi\,'(\Omega)
        (\dot{\Psi}\Sigma\Psi + \Psi\Sigma\dot{\Psi})
    )
     =
    \Tr
    (
        (\dot{\Psi}\Psi^{-1} + \Psi^{-1}\dot{\Psi})
        \chi\,'(\Omega)\Omega
    )
$, where the commutativity of
$\Omega$ and $\chi\,'(\Omega)$ is used. In a similar vein, the
variation $
    \delta\Tr\chi(\Omega)
    =
    \Tr
    (
        \chi\,'(\Omega)
        \Psi (\delta\Sigma)\Psi
    )
    =
    \Tr
    (
        \Psi
        \chi\,'(\Omega)
        \Psi \delta\Sigma
    )
$ with respect to  $\Sigma$ for a fixed $T$ yields  (\ref{chidiffdiff}).

\section*{B. Separation of variables for analytic functions of matrices}
\renewcommand{\theequation}{B\arabic{equation}}
\setcounter{equation}{0}

The following lemma provides a separation-of-variables
technique for analytic functions of matrices in
Section~\ref{sec:trace_analytic_correction}.

\begin{lemma}
\label{lem:matrix_separation} Let $\chi_1$, $\chi_2$ be functions
of a complex variable, analytic in a neighbourhood of
and real-valued on $\mR_+$. Let $G_1$, $G_2$ be
$\mS_n$-valued functions of an independent variable $T$  with
$\Tr G_1(T_0) \ne 0$ for some value $T_0$ of $T$. Suppose that
\begin{equation}
\label{constant}
    \sum_{k = 1}^{2}
    \Tr
    (
        G_k(T) \chi_k(\Omega)
    )
    =
    R(T)
    \qquad
    {\rm for\ all}\
    \Omega\in \mS_n^+\
    {\rm and}\
    T,
\end{equation}
where $R(T)$ is a function of $T$ only.  Then $\chi_1$, $\chi_2$ are
affinely dependent in their common analyticity domain:
\begin{equation}
\label{chi_aff_dep}
    \chi_1 + \lambda \chi_2
    \equiv
    \tau,
\end{equation}
where  $\lambda$ and $\tau$ are real constants. If, in
addition to the previous assumptions,
$\chi_1$ or $\chi_2$ is nonconstant,   then the pair
$(\lambda,\tau)$ is unique, and
\begin{equation}
\label{collinear}
    G_2
    \equiv
    \lambda G_1,
    \qquad
    R
    \equiv
    \tau \Tr G_1.
\end{equation}
\end{lemma}
\begin{proof}
By considering (\ref{constant}) for scalar
matrices $\Omega = \omega I_n$, with $\omega \in \mR_+$,
it follows that
\begin{equation}
\label{affcomb}
    \sum_{k=1}^{2}
    \chi_{k}(\omega)
    \Tr G_k(T)
    =
    R(T)
    \qquad
    {\rm for\ all}\
    \omega \> 0\
    {\rm and}\
    T.
\end{equation}
Dividing both sides of (\ref{affcomb}) for $T = T_0$ by
$\Tr G_1(T_0)\ne 0$ and introducing the ratios
$$
    \lambda
    :=
    \Tr G_2(T_0)/\Tr G_1(T_0),
    \qquad
    \tau
    :=
    R(T_0)/\Tr G_1(T_0)
$$
yields the affine dependence (\ref{chi_aff_dep}) of the functions
$\chi_1$ and $\chi_2$ on $\mR_+$ which extends to their common
analyticity domain by the identity theorem of complex analysis. Now,
let $\chi_1$ or $\chi_2$ be
nonconstant. Then (\ref{chi_aff_dep}) determines the pair
$(\lambda,\tau) \in \mR^2$ uniquely. Therefore, the following
implication holds for any  $\mu := (\mu_1, \mu_2, \mu_3)\in \mR^3$:
\begin{equation}
\label{coll}
    \mu_1 \chi_1
    +
    \mu_2 \chi_2
    \equiv
    \mu_3
    \Longrightarrow
    \mu_2
    =
    \lambda \mu_1,\,
    \mu_3 = \tau \mu_1.
\end{equation}
For any $\omega \> 0$ and $u \in \mR^n$
with $|u| = 1$, consider a matrix
$    \Omega
    :=
    \omega uu^{\rT}
$. Its eigenvalues are $\omega$ (with the eigenvector $u$) and 0 (with
the  eigenspace $u^{\bot}$, the orthogonal complement of $u$ in $\mR^n$). The spectral decomposition of $\Omega$
yields
$    \chi
    (    \omega uu^{\rT}
    )
    =
    \chi(\omega)
    uu^{\rT}
        +
    \chi(0)
    (
        I_n - uu^{\rT}
    )
$, and (\ref{constant}) takes the form
\begin{equation}
\label{sep}
    \sum_{k=1}^{2}
        u^{\rT}G_k(T) u\,
        \chi_k(\omega)
    =
    R(T)
        -
    \sum_{k=1}^{2}
        \chi_k(0)
        (
            \Tr G_k(T) - u^{\rT} G_k(T) u
        ),
    \qquad
    \omega \> 0.
\end{equation}
Here, all $\omega$-independent terms are moved
to the right-hand side. By considering (\ref{sep}) for fixed but
otherwise arbitrary values of $u$ and $T$ and recalling
(\ref{coll}), it follows that
\begin{align}
\label{uMu}
    u^{\rT}
    G_2(T)
    u
     & =
    \lambda
    u^{\rT}
    G_1(T)
    u,\\
\label{RG}
    R(T)
    -
    \sum_{k=1}^{2}
        \chi_k(0)
        (
            \Tr G_k(T) - u^{\rT} G_k(T) u
        )
         & =
        \tau
        u^{\rT}
        G_1(T)u.
\end{align}
Since the unit vector $u$ is arbitrary and the matrices $G_1(T)$ and $G_2(T)$ are both symmetric, then
(\ref{uMu}) implies the first of the relations (\ref{collinear}). Combining the latter with
(\ref{chi_aff_dep}) gives
\begin{align}
\nonumber
    \sum_{k=1}^{2}
        \chi_k(0)
        (
            \Tr G_k(T) - u^{\rT} G_k(T) u
        )
         & =
        (\chi_1(0) + \lambda \chi_2(0))
                (
            \Tr G_1(T) - u^{\rT} G_1(T) u
        )\\
\label{sumchi}
 & =
        \tau
        (
            \Tr G_1(T) - u^{\rT} G_1(T) u
        ).
\end{align}
Substitution of (\ref{sumchi}) into (\ref{RG}) yields
\begin{align}
\nonumber
    R(T)
     & =
    \tau u^{\rT}G_1(T)u
    +
    \sum_{k=1}^{2}
        \chi_k(0)
        (
            \Tr G_k(T) - u^{\rT} G_k(T) u
        )\\
    \nonumber
     & =
    \tau u^{\rT}G_1(T)u
    +
        \tau
        (
            \Tr G_1(T) - u^{\rT} G_1(T) u
        )
        =
        \tau \Tr G_1(T)
\end{align}
which, by the arbitrariness of $T$, establishes the second of the
relations (\ref{collinear}).
\end{proof}
\end{document}